\documentclass[11pt]{amsart}

\usepackage{epigamath-voisin}

\usepackage[english]{babel}

\numberwithin{equation}{section}

\usepackage{amsmath,latexsym,amsfonts,amssymb,amsthm}
\usepackage{mathrsfs}
\usepackage{graphicx,color}
\usepackage[alphabetic,initials, short-journals,lite]{amsrefs}
\usepackage[all,cmtip]{xy}
\usepackage{bbm}
\usepackage{tabu}
\usepackage[shortlabels]{enumitem}
\usepackage{tikz}
\usepackage{tikz-cd}

\newtheorem{prop}{Proposition}[section]
\newtheorem{thm}[prop]{Theorem}

\newtheorem{cor}[prop]{Corollary}

\newtheorem{lem}[prop]{Lemma}

\theoremstyle{definition}

\newtheorem{defn}[prop]{Definition}
\newtheorem{assump}[prop]{Assumption}

\newtheorem*{claim*}{Claim}

\theoremstyle{remark}
\newtheorem{expl}[prop]{Example}
\newtheorem{rem}[prop]{Remark}

\newcommand{\bR}{\mathbb{R}}

\newcommand{\bQ}{\mathbb{Q}}

\newcommand{\bN}{\mathbb{N}}

\newcommand{\bk}{\mathbbm{k}}

\newcommand{\cO}{\mathcal{O}}

\newcommand{\cI}{\mathcal{I}}

\newcommand{\cF}{\mathcal{F}}

\newcommand{\fa}{\mathfrak{a}}

\DeclareMathOperator{\lct}{lct}

\DeclareMathOperator{\vol}{vol}
\DeclareMathOperator{\ord}{ord}

\DeclareMathOperator{\Gr}{Gr}

\DeclareMathOperator{\Ex}{Ex}
\renewcommand{\div}{\operatorname{div}}
\DeclareMathOperator{\Sym}{Sym}

\DeclareMathOperator{\Proj}{Proj}

\DeclareMathOperator{\Bs}{Bs}

\makeatletter
\renewcommand{\BibLabel}{%
    \Hy@raisedlink{\hyper@anchorstart{cite.\CurrentBib}\hyper@anchorend}%
    [\thebib]%
}
\makeatother

\BibSpec{arXiv}{%
  +{}{\PrintAuthors}{author}
  +{,}{ \textit}{title}
  +{,}{  }{eprint}
  +{}{ \parenthesize}{date}
  +{.}{}{transition}
}

\setlist[enumerate,1]{label={\rm(\arabic*)}, ref={\rm\arabic*}} 
\newcommand{\lra}{\longrightarrow}

\YearArticle{2023} \EpigaArticleNr{7} \ReceivedOn{November 1, 2022}
\InFinalFormOn{June 10, 2023}
\AcceptedOn{June 27, 2023}

\title{K-stability for varieties with a big anticanonical class}
\titlemark{K-stability for varieties with a big anticanonical class}

\author{Chenyang Xu}
\address{Department of Mathematics, Princeton University, Princeton, NJ 08544, USA}
\email{chenyang@princeton.edu}

\authormark{C.~Xu}

\AbstractInEnglish{We extend the algebraic K-stability theory to projective klt pairs with a big anticanonical class. While in general such a pair could behave pathologically, it is observed in this note that the K-semistability condition will force them to have a klt anticanonical model, whose stability property is the same as that of the original pair.}

\MSCclass{14J45, 14E99, 32Q20}

\KeyWords{$K$-stability, finite generation, log Fano varieties}

\acknowledgement{The author is partially supported by NSF Grant DMS-2153115, DMS-2139613, DMS-2201349 and Simons Investigator Program.}

\dedication{To Claire Voisin on the occasion of the conference in her honor}

\begin{document}



\maketitle

\begin{prelims}

\DisplayAbstractInEnglish

\bigskip

\DisplayKeyWords

\medskip

\DisplayMSCclass







\end{prelims}


\newpage

\setcounter{tocdepth}{1}

\tableofcontents


\section{Introduction}
There has been tremendous progress in algebraic K-stability theory of log Fano pairs (see \cite{Xu-survey} for a survey of the topic).
In the recent works \cite{DZ-bigness} and \cite{DR-bigness}, the K\"ahler--Einstein problem is considered for a K\"ahler manifold $(X,\omega)$ such that $-K_X$ is big. More precisely, in \cite{DZ-bigness} the authors prove a transcendental Yau--Tian--Donaldson theorem for twisted big K\"ahler--Einstein metrics. As a consequence of their result, in the algebraic setting,  uniform K-stability of $X$ with a big anticanonical class implies the existence of a K\"ahler--Einstein metric. In this note we want to show that the K-stability theory in this setting, \textit{i.e.} a projective klt pair with a big anticanonical class, essentially follows from the original (log) Fano case. 

In general,  there could be pathological examples in  projective varieties $X$ with a big anticanonical class $-K_X$; \textit{e.g.} the anticanonical ring $R(X,-K_X)=\bigoplus_{m\in \bN}H^0(X,-mK_X)$ is not necessarily finitely generated (see Example~\ref{example-non finitely generated}). However, we will show that the K-stability condition implies that $X$ is of log Fano type.

\begin{thm}\label{t-main}
Let $(X,\Delta)$ be a klt projective pair with $-K_X-\Delta$ big. Assume $\delta(X,\Delta)\ge 1$. Then there exists an effective $\bQ$-divisor $\Gamma$ such that $(X,\Delta+\Gamma)$ is a log Fano pair, \textit{i.e.}\ $(X,\Delta+\Gamma)$ is klt and $-K_X-\Delta-\Gamma$ is ample. In particular, 
$$R(X,-r(K_X+\Delta))=\bigoplus_{m\in r\cdot\bN}H^0(X,-m(K_X+\Delta))$$
is finitely generated for any $r$ such that $r(K_X+\Delta)$ is Cartier.
\end{thm}
Here $\delta(X,\Delta)$ is defined in the exactly same fashion as in the case when $-K_X-\Delta$ is ample (see \cites{FO-basistype,BJ-delta}), \textit{i.e.}\
\[
\delta(X,\Delta)=\inf_E\frac{A_{X,\Delta}(E)}{S_{X,\Delta}(E)}.
\]
For the stronger and more precise statement, see Theorem~\ref{thm-log fano pair}. We note that the above finite generation is asked in \cite{DR-bigness}.

The above observation makes it possible to use existing birational geometry techniques to study K-stability questions for $X$ with a big anticanonical class. In fact, without too much difficulty, it reduces K-stability questions for $(X,\Delta)$ to K-stability questions for its anticanonical model $(Z,\Delta_Z)$, as we can see from the following statement.

\begin{thm}\label{t-maintheorem}
Let $(X,\Delta)$ be a klt projective pair with $-K_X-\Delta$  big. Assume $R=\bigoplus_{m\in r\cdot \bN} H^0(-m(K_X+\Delta))$ is finitely generated, and denote by $(Z,\Delta_Z)$
the anticanonical model. Then 
$(X,\Delta)$ is K-semistable $($resp.\ K-stable, uniformly K-stable$)$ if and only $(Z,\Delta_Z)$ is  K-semistable $($resp.\ K-stable, uniformly K-stable$)$.   In particular, uniform K-stability of\, $(X,\Delta)$ is the same as K-stability of\, $(X,\Delta)$.
\end{thm} 

\begin{rem}In \cite{DR-bigness}, Ding stability notions for a projective klt pair $(X,\Delta)$ with big $-K_X-\Delta$ are developed. If one assumes $R=\bigoplus_{m\in r\cdot \bN} H^0(-m(K_X+\Delta))$ is finitely generated and denotes by $(Z,\Delta_Z)$ the anticanonical model, then one can show a similar statement to Theorem~\ref{t-maintheorem}; \textit{i.e.}\ the Ding stability notions for $(X,\Delta)$ are equivalent to the notions for $(Z,\Delta_Z)$. 
\end{rem}

\subsubsection*{Notation and Convention} Throughout this paper, we work over an algebraically closed field $\bk$ of characteristic $0$. We follow the standard terminology from \cites{KM98, Kol13}.

For a normal log pair $(X,\Delta)$ such that $K_X+\Delta$ is $\bQ$-Cartier and a divisor $E$ over $X$, we denote by $A_{X,\Delta}(E)$ the log discrepancy of $E$ with respect to $(X,\Delta)$.

We say a klt projective pair $(X,\Delta)$ is \emph{log Fano} if $(X,\Delta)$ is klt and $-K_X-\Delta$ is ample, and a klt projective pair $(X,\Delta)$ is \emph{of log Fano type} if there  exists an effective $\bQ$-divisor $D$ such that $(X,\Delta+D)$ is a log Fano pair.

We say an effective  $\bQ$-divisor $\Gamma$ on a  projective log pair $(X,\Delta)$ is an {\it $N$-complement} for a positive integer $N$ if $N(K_X+\Delta+\Gamma)\sim0$ and $(X,\Delta+\Gamma)$ is log canonical. A {\it $\bQ$-complement} is an $N$-complement for some $N$.

\subsection*{Acknowledgments}
We would like to thank Tam\'as Darvas for his  lecture that sparked the author's interest in this question and useful discussions afterwards. We would like to thank Ziquan Zhuang for communicating Example~\ref{example-non finitely generated} to us. We want to thank R\'emi Reboullet and Ruadha\'i Dervan, as well as the anonymous referee, for helpful comments. The project was initiated when the author attended Northwestern-UIC Complex Geometry Seminar, to the organizer of which the author owes his gratitude.

\section{\texorpdfstring{$\boldsymbol{S}$}{S}-invariants}\label{s-prelim}


Let $(X,\Delta)$ be an $n$-dimensional projective normal pair such that $-K_X-\Delta$ is big. For any prime divisor $E$ which appears on a birational model $\mu\colon Y\to X$, the $S$-invariant is defined as
$$
S_{X,\Delta}(E):=\frac{1}{\vol(-K_X-\Delta)}\int^{\infty}_{0}\vol(-\mu^*(K_X+\Delta)-tE)\ dt.
$$

\begin{defn}
If $(X,\Delta)$ is klt, we define 
$$\delta(X,\Delta):=\inf_E\frac{A_{X,\Delta}(E)}{S_{X,\Delta}(E)},$$
where $E$ runs through all valuations over $(X,\Delta)$. We say $(X,\Delta)$ is uniformly K-stable (resp.\ K-semistable), if $\delta(X,\Delta)>1$ (resp.\ $\delta(X,\Delta)\ge 1$).  We say $(X,\Delta)$ is K-stable if $A_{X,\Delta}(E)>S_{X,\Delta}(E)$ for any $E$ over $X$.
\end{defn}
\begin{rem}
When $(X,\Delta)$ is log Fano, the equivalence between this way of defining K-stability notions using valuations and the original one using test configurations, called the Fujita--Li criterion, is proved in \cite{Fujita-valuative}, \cite{Li-valuative}  and \cite{BX-uniqueness}.
For $(X,\Delta)$ with a big anticanonical class, the current definition is formulated in \cite{DZ-bigness}.
\end{rem}

\begin{rem}Theorem~\ref{t-maintheorem} says K-stability is indeed the same as uniform K-stability. For a log Fano pair, this is proved in \cite{LXZ-HRFG} (see \cite{XZ-localHRFG} for a different proof). 
\end{rem}

Fix $m\in r\cdot \bN$, let $R_m=H^0(X,-m(K_X+\Delta))$, and assume $N_m:=\dim H^0(X,-m(K_X+\Delta))>0$.
Following \cite{FO-basistype}, we say a $\bQ$-divisor $D$ is an \emph{$m$-basis type divisor} if 
\[
\frac{1}{m\cdot N_m}\ord_E\left({\div}(s_1)+\cdots +{\div}(s_{N_m})\right)
\]
for a basis $\{s_1,\ldots,s_{N_m}\}$ of $R_m$. In particular, $D\sim_{\bQ}-K_X-\Delta$. 

We  define $S_{X,\Delta,m}(E)$ (or $S_{m}(E)$ if $(X,\Delta)$ is clear) for any $E$ over $X$ as follows: $E$ yields a decreasing filtration $\cF^{\lambda}_E$ $(\lambda\in \bR)$ on $R_m:=H^0(X,-m(K_X+\Delta))$ by
$$
\cF^{\lambda}_E R_m=\left\{\,s\in H^0(X,-m(K_X+\Delta)) \ | \ \ord_E(s)\ge \lambda\,\right\},
$$
and 
\[
S_m(E)=\frac{1}{m\cdot N_m}\ord_E\left({\div}(s_1)+\cdots +{\div}(s_{N_m})\right)
\]
for any basis $\{s_1,\ldots,s_{N_m}\}$ of $R_m$ compatible with $\cF^{\lambda}_E$ $(\lambda\in \bR)$. Here the basis
$\{s_1,\ldots,s_{N_m}\}$ is compatible with $\cF^{\lambda}_E$ $(\lambda\in \bR)$ if for any $\lambda$, all the elements $s_i$ contained in  $\cF^{\lambda}_ER_m$ span $\cF^{\lambda}_ER_m$. Then
\[
S_m(E)=\frac{1}{m\cdot N_m}\sum_{\lambda\in \bN}\lambda\cdot\dim {\Gr}_E^{\lambda} R_m,
\]
where  $ {\Gr}_E^{\lambda} R_m:=\cF_E^{\lambda}R_m/\cF_E^{\lambda+1}R_m$.

We also define
\[
\delta_m(X,\Delta):=\inf_E\frac{A_{X,\Delta}(E)}{S_m(E)}.
\]

The following are basic properties proved in \cite{BJ-delta}.

\begin{thm}\label{t-Sm and S}
Keep the notation as above. 
\begin{enumerate}
\item\label{thm-1} We have $\lim_{m\to \infty}S_m(E)=S(E)$.
\item\label{thm-2} For any $\varepsilon>0$, there exists an $m_0$ such that for any $E$ over $X$ and $m\ge m_0$ with $m\in r\cdot\bN$, 
\[
S_m(E)\le (1+\varepsilon)S(E).
\]
\item\label{thm-3} We have $\delta_m(X,\Delta)=\inf_D\lct(X,\Delta;D)$, where $D$ runs through all $m$-basis type divisors. 
\item\label{thm-4} We have $\lim_{m\to \infty} \delta_m(X,\Delta)=\delta(X,\Delta)$.
\end{enumerate}
\end{thm}

\begin{proof}Statement~\eqref{thm-1} follows from \cite[Lemma 2.9]{BJ-delta} and~\eqref{thm-2} from \cite[Corollary 2.10]{BJ-delta}.
Statement~\eqref{thm-3} is \cite[Proposition 4.3]{BJ-delta}, and~\eqref{thm-4} is \cite[Theorem 4.4]{BJ-delta}.
\end{proof} 

We can consider more general filtrations.
\begin{defn}\label{defn-filtration}
By a (linearly bounded) filtration $\cF$ on $R(X, -r(K_X+\Delta)) =\bigoplus_{m\in r\cdot \bN}R_m$, we mean the data of a family
$\cF^{\lambda}R_m \subseteq R_m$
of $\bk$-vector subspaces of $R_m$ for $m \in r\cdot \bN$ and $\lambda \in \bR$, satisfying
\begin{enumerate}
\item $\cF^{\lambda'}R_m\subseteq \cF^{\lambda}R_m $ when $\lambda \ge \lambda'$;
\item $\cF^{\lambda}R_m=\bigcap_{\lambda'<\lambda}\cF^{\lambda'}R_m$ for any $\lambda$;
\item there exist $e_-, e_+\in \bR$ such that $\cF^{me_-}R_m=R_m$ and $\cF^{me_+}R_m=0$ for  any $m$;
\item $\cF^{\lambda}R_m\cdot\cF^{\lambda'}R_{m'}\subseteq \cF^{\lambda+\lambda'}R_{m+m'}$.
\end{enumerate}
\end{defn}
For any filtration $\cF$ on $R$, we can  define $S_m(\cF)$ and $S(\cF)$ as in \cite[Sections~2.5 and~2.6, pp.~15--16]{BJ-delta},
and we have
\begin{equation}
 \lim_{m\to \infty}S_m(\cF)\lra S(\cF); 
\end{equation}
see \cite[Lemma 2.9]{BJ-delta}.


\begin{lem}\label{lem-1/n+1}
If $A$ is an effective ample $\bQ$-divisor on $X$ such that $-K_X-\Delta-A$ is pseudoeffective, then $S_{X,\Delta}(A)\ge \frac{1}{n+1}$.
\end{lem}

\begin{proof}
Since $-K_X-\Delta-A$ is pseudoeffective, for any $t\ge 0$, we have  
$$\vol(-K_X-\Delta-tA)\ge \vol((1-t)(-K_X-\Delta)).$$ 
Thus
\begin{align*}\pushQED{\qed}
S(A) &=  \frac{1}{\vol(-K_X-\Delta)}\int^{+\infty}_{0}\vol(-K_X-\Delta-tA)dt\\
  &\ge  \frac{1}{\vol(-K_X-\Delta)}\int^1_0\vol((1-t)(-K_X-\Delta))dt\\
&= \frac{1}{(n+1)}.\hfill\qedhere \popQED
\end{align*}
\renewcommand{\qed}{}   
 \end{proof}

\section{Finite generation}

\subsection{$\bQ$-complements and finite generation}

For a $\bQ$-divisor $D$ with $|rD|\neq 0$ and any $m\in \bN$, we denote by $\Bs(|mrD|)$ the base ideal. We can define the log canonical threshold of the asymptotic linear series as follows:
\[
\lct(X,\Delta;\|-K_X-\Delta\|):=\sup_\ell\, \lct\left(X,\Delta;\frac{1}{\ell r} \Bs| \ell r(-K_X-\Delta)|\right).
\]
We can define a sequence of multiplier  ideals
\[
 \mathcal{I}\left(X,\Delta; \frac{1}{r}\Bs(|rD|)\right) \subseteq  \mathcal{I}\left(X,\Delta; \frac{1}{2r}\Bs(|2rD|)\right) \subseteq\cdots \subseteq  \mathcal{I}\left(X,\Delta; \frac{1}{\ell!r}\Bs(|\ell!rD|)\right)\subseteq \cdots .
\]
By the ascending chain condition of ideals, this sequence will stabilize. We denote the maximal element by $ \mathcal{I}(X,\Delta; \|-K_X-\Delta\|)$ and call it the \emph{asymptotic multiplier ideal sheaf of $D$}.  For more background, see \cite[Section 11.1]{Lazasfeldbook2}. Recall that for any ideal $\fa\subseteq \cO_X$, we have $\lct(X,\Delta;\fa)>1$ if and only if $\cI(X,\Delta;\fa)>1 $. 

\begin{lem}\label{lem-multiplier trivial}
Assume $(X,\Delta)$ is a projective pair with $-K_X-\Delta$  big.
  If $$\lct(X,\Delta;\|-K_X-\Delta\|)>1 \quad   (\mbox{or equivalently }\ \mathcal{I}(X,\Delta; \|-K_X-\Delta\|)=\mathcal{O}_X),$$ then $(X,\Delta)$ is of log Fano type.
\end{lem}

\begin{proof}
From the assumption, there exists an effective $\bQ$-divisor $D\sim_{\bQ}-(K_X+\Delta)$ such that $(X,\Delta+D)$ is klt. Since $D$ is big, then $D\sim_{\bQ}A+E$ for an ample $\bQ$-divisor $A$ and an effective $\bQ$-divisor $E$.  Set
$$\Gamma:=(1-\varepsilon)D+\varepsilon E$$
for $0<\varepsilon\ll 1$; then $(X,\Delta+\Gamma)$ is klt, and $-K_X-\Delta-\Gamma\sim \varepsilon A$ is ample. Thus $(X,\Delta)$ is of log Fano type.
\end{proof}

\begin{defn}
For any projective pair $(X,\Delta)$, we define the constant $a(X,\Delta)$ by 
\begin{equation}
\begin{split}
a(X,\Delta)=  \sup_{t\in \bR}
    \left\{
    \begin{aligned}
         & \mbox{there exists an ample divisor $A$ such that $A-t(K_X+\Delta)$} \\
         & \mbox{\ \ \ \ is ample and $-K_X-\Delta-A$ is pseudoeffective }
    \end{aligned}
    \right\}.
\end{split}
\end{equation}
\end{defn}
If $-K_X-\Delta$ is big, then $a(X,\Delta)>0$;  if $-K_X-\Delta$ is ample, then $a(X,\Delta)=+\infty$.

\begin{assump}\label{assump-delta}
Let $(X,\Delta)$ be an $n$-dimensional klt projective pair with $-K_X-\Delta$  big. Assume 
\begin{equation}\label{eq lower bound}
\delta(X,\Delta)>\frac{n+1}{n+1+a_0},\quad\mbox{where $a_0=a(X,\Delta)$}.
\end{equation}

\end{assump}

Now we can show the following.

\begin{thm}\label{thm-log fano pair}
 Let $(X,\Delta)$ satisfy Assumption~\ref{assump-delta}; then $(X,\Delta)$ is of log Fano type. In particular, any Cartier divisor $E$ on $X$ satisfies that $R(X,E):=\bigoplus_{m\in \bN}H^0(X,mE)$ is finitely generated.
\end{thm}

\begin{proof}Let us first prove this when $\delta(X,\Delta)>1$ as it is quite straightforward.
By Theorem~\ref{t-Sm and S}, we know that for a sufficiently large $m$ and any $m$-basis type divisor $D$,
$$\lct(X,\Delta;D)\ge \delta_m(X,\Delta)>1.$$ Thus we can apply Lemma~\ref{lem-multiplier trivial}.

\medskip

In the general case, we may assume $\delta(X,\Delta)\le 1$, and we need some perturbation argument.  By our definition of $a(X,\Delta)$, for any $t\in (0,a(X,\Delta))$, there exists an ample $\bQ$-divisor $A$ such that
\[
-K_X-\Delta-A\sim_{\bQ}E_1\quad\mbox{and}\quad A-t(K_X+\Delta)\sim_{\bQ}A_0,
\]
where $E_1$ is an effective $\bQ$-divisor  and $A_0$ is an ample $\bQ$-divisor.  Moreover, by \eqref{eq lower bound} we may assume 
\begin{equation}\label{eq-t close to a}
1-\delta(X,\Delta)< \frac{t}{n+1}\delta(X,\Delta).
\end{equation}

 Fix $m_0\in \bN$ such that $|m_0A|$ is base-point-free. Then for any prime divisor $H\in |m_0A|$, by Lemma~\ref{lem-1/n+1},
\begin{align*}
S(H) &= \frac{1}{\vol(-K_X-\Delta)}\int\vol(-K_X-\Delta-tH)dt\\
  &= \frac{1}{m_0}S_{X,\Delta}(A) \\
  &\ge \frac{1}{m_0(n+1)}.
\end{align*}

We can choose an $m$-basis type $\bQ$-divisor $D_m$ compatible with  $H$, so we can write $D_m=F_m+b_mH$, where 
\begin{equation}\label{eq-estimate bm}
\lim_{m\to \infty} b_m=\lim_{m\to \infty} S_m(H)=S(H) \ge \frac{1}{m_0(n+1)}.
\end{equation}

By \eqref{eq-t close to a}, \eqref{eq-estimate bm}, and the equality $\lim_m\delta_m(X,\Delta)=\delta(X,\Delta)$, we can find a sufficiently large $m$ and a positive $\delta'$  such that $\delta'<\min\{\delta_m(X,\Delta),1\}$ and
 \begin{equation}\label{eq-estimate delta}
 1-\delta'< tm_0b_m\delta'.
 \end{equation}
Then $(X,\Delta+\delta'F_m)$ is klt, as $(X,\Delta+\delta'D_m)$ is klt and $D_m=F_m+b_mH$. Moreover,
\[
-K_X-\Delta-\delta'F_m\sim_{\mathbb{Q}} -(1-\delta')(K_X+\Delta)+\delta'b_mH,
\]
which implies $(X,\Delta+\delta'F_m)$ is a log Fano pair since
\begin{align*}
  -(1-\delta')(K_X+\Delta)+\delta'b_mH & \sim_{\bQ}(1-\delta')\left(-(K_X+\Delta)+\frac{1}{t}A\right)+\left(\delta'b_mm_0-\frac{1-\delta'}{t}\right)A \\
  &\sim_{\bQ}\frac{1-\delta'}{t}A_0+\left(\delta'b_mm_0-\frac{1-\delta'}{t}\right)A
  \end{align*}
is ample by \eqref{eq-estimate delta}.

The last statement then follows from \cite{BCHM}.
\end{proof}

\begin{cor}\label{cor-logfanomodel}
Let $(X,\Delta)$ satisfy  Assumption~\ref{assump-delta}. Let $r(K_X+\Delta)$ be Cartier and $Z:={\Proj}~R(X,-r(K_X+\Delta))$. Denote by $\Delta_Z$ the birational transform of $\Delta$ on $Z$; then $(Z,\Delta_Z)$ is a log Fano pair.
\end{cor}

\begin{proof}
We know $f\colon X\dasharrow Z$ is a birational contraction; \textit{i.e.}\ ${\Ex}(f^{-1})$ does not contain any divisor, and $f_*(K_X+\Delta)=K_Z+\Delta_Z$ is antiample.

It follows from Theorem~\ref{thm-log fano pair} that there exists a $\bQ$-complement $\Gamma$ for $(X,\Delta)$ such that $(X,\Delta+\Gamma)$ is klt. Then $(Z,\Delta_Z+f_*\Gamma)$ is klt as the pullbacks of $K_Z+\Delta_Z+f_*\Gamma$ and $K_X+\Delta+\Gamma$ on a common resolution are equal. So $(Z,\Delta_Z)$ is klt. 
\end{proof}

\subsection{K-stability of the anticanonical model}
Let $(X,\Delta)$ be a projective log pair with big $-K_X-\Delta$. Let $(Z,\Delta_Z)$ be its anticanonical model; \textit{i.e.}\ $Z={\Proj}~R(X,-r(K_X+\Delta))$, and $\Delta_Z$ is the birational transform of $\Delta$ on $Z$.
Let $Y$ be a common resolution. 

    \begin{equation}\label{ef-resolvetest} 
        \xymatrix{
    &    Y\ar[dl]_{\mu}\ar[dr]^{\pi}  &\\
       {(X,\Delta)}  \ar@{..>}[rr]^{f}& & (Z,\Delta_Z).
        }  
    \end{equation}
Then 
$$\pi^*(K_Z+\Delta_Z)-\mu^*(K_X+\Delta)=B\ge 0$$

\begin{lem}Let $(X,\Delta)$ satisfy Assumption~\ref{assump-delta}.
Then for any prime divisor $E$ over $X$,
\[
A_{X,\Delta}(E)=A_{Z,\Delta_Z}(E)+\ord_E(B) \mbox{ and   }S_{X,\Delta}(E)=S_{Z,\Delta_Z}(E)+\ord_E(B) .
\]
\end{lem}
\begin{proof}
For the log discrepancy function, this follows directly from the definition. Since 
$$|\mu^*(-m(K_X+\Delta))|=|\pi^*(-m(K_Z+\Delta_Z))|+mB,$$
we have $S_{X,\Delta,m}(E)=S_{Z,\Delta_Z,m}(E)+\ord_E(B)$. Therefore, the same is true for the $S$-function.
\end{proof}

\begin{lem}\label{lem-compareZandB}
If\, $(Z,\Delta_Z)$ is klt, there exists a $t>0$ depending on $Z$ $($but not $E)$ such that for any divisor $E$ over $X$
\[
A_{Z,\Delta_Z}(E)\ge t \cdot \ord_E(B). 
\]
\end{lem}

\begin{proof}
Since $(Z,\Delta_Z)$ is klt, we know that there exists a $t>0$ such that if we write $\pi^*(K_Z+\Delta_Z)=K_Y+\Delta_1$, then $(K_Y+\Delta_1+ t B)$ is sub-lc for some $t>0$; \textit{i.e.}\ for any $E$,
\begin{equation*}\pushQED{\qed}
  A_{Z,\Delta_Z}(E)\ge t \cdot \ord_E(B).\qedhere \popQED
	\end{equation*}
\renewcommand{\qed}{}     
\end{proof}

\begin{proof}[Proof of Theorem~\ref{t-maintheorem}]
Since 
\[
\delta(X,\Delta)=\inf_E\frac{A_{Z,\Delta_Z}(E)+\ord_E(B)}{S_{Z,\Delta_Z}(E)+\ord_E(B)},
\]
it is clear that $\delta(X,\Delta)\ge 1$ if and only if $A_{Z,\Delta_Z}(E)\ge S_{Z,\Delta_Z}(E)$, \textit{i.e.}\ $(Z,\Delta)$ is klt and $\delta(Z,\Delta_Z)\ge 1$. Moreover, 
$$A_{X,\Delta}(E)=A_{Z,\Delta_Z}(E)+\ord_E(B)>S_{Z,\Delta_Z}(E)+\ord_E(B)=S_{X,\Delta}(E)$$ if and only if $A_{Z,\Delta_Z}(E)>S_{Z,\Delta_Z}(E)$.

Assume $\delta(X,\Delta)>1$; then $\delta(Z,\Delta_Z)\ge \delta(X,\Delta)$. Conversely, if $\delta(Z,\Delta_Z)> 1$, an easy calculation shows that
$$\delta(X,\Delta)\ge \frac{\delta(Z,\Delta_Z)(t+1)}{\delta(Z,\Delta_Z)+t}>1,$$
where $t$ is the constant from Lemma~\ref{lem-compareZandB}.
\end{proof}

\begin{expl}\label{example-non finitely generated}
This example has appeared in several works to present pathological phenomena, see \textit{e.g.} \cite{Gongyo-weakfano}: Let $S$ be the  blowup of $\mathbb{P}^2$ at nine very general points. Then $-K_S$ is known to be nef but not semiample. In fact, there will be a unique cubic curve passing through these nine points, and if we denote by $E$ its birational transform on $S$, then for any $m\in \bN$, $|-mK_S|$ has one element $mE$. 

Let $H$ be an ample Cartier divisor on $S$ and $X=\mathbb{P}_S(E)$, where $E:=\mathcal{O}_S+\mathcal{O}_S(H)$. Denote by $\pi\colon X\to S$ the natural morphism. We claim $-K_X$ is big. In fact, since
$$
\omega_{X/S}=\wedge^2\mathcal{O}_{\mathbb{P}(E)}(-2),
$$
we have
\begin{align*}
H^0(\cO_X(-mK_X)) &= H^0(S,\pi_*(\mathcal{O}_X(-mK_X)))\\
&= H^0(S, {\Sym}^{2m}(E)\otimes(\wedge^2E)^{\otimes -m}\otimes \omega^{\otimes m}_S)\\
&= H^0\left(S,  \left(\bigoplus^{2m}_{i=0}  \mathcal{O_S}(iH)\right) \otimes \mathcal{O}_S(-mH-mK_S) \right)\\
      &= H^0\left(S, \bigoplus^m_{i=0}\mathcal{O}_S(iH-mK_S) \right),
\end{align*}
and since $-K_S\sim E$ is nef, we have
\begin{align*}
\vol_X(-K_X) &= 6\int^1_{0}\frac{1}{2}(tH-K_S)^2=3\int^1_{0}(t^2H^2-2tH(-K_S))\\
                 &= H^2+3H\cdot (-K_S)>0.
\end{align*} 
 However, the algebra $\bigoplus_{i\le m}H^0(iH-mK_S)$ is not finitely generated,
since
$$\sum_{1\le j\le m-1}H^0(\cO_S(-jK_S))\otimes H^0(\cO_S(H-(m-j)K_S))\lra H^0(\cO_S(H-mK_S))$$ is not surjective for any $m$. Thus we need generators from $H^0(\cO_S(H-mK_S))$ for every $m$.

By Theorem~\ref{t-main}, we know $\delta(X)<1$. Here we give a direct verification of this. We denote by $Y\subseteq X$ the section given by
$$E= \mathcal{O}_S\oplus  \mathcal{O}_S(H)\lra \mathcal{O}_S.$$
Then similarly to before, we have
\[
H^0(\cO_X(-mK_X-m_0Y)) =  H^0\left(S,  \left(\bigoplus^{2m}_{i=m_0}  \mathcal{O_S}(iH)\right) \otimes \mathcal{O}_S(-mH-mK_S) \right),
\]
where we follow the convention that if $m_0> 2m$, then the direct sum is 0. Hence a direct calculation implies
$$
\vol(-K_X-tY)\ =
\begin{cases}
H^2+3H\cdot (-K_S) \hspace{40mm}\mbox{if } t\le 1,\\
(2-t)((t^2-t+1)H^2+3t H\cdot (-K_S)) \hspace{4mm}\mbox{if } 1\le t\le 2.
\end{cases}
$$
By an elementary calculation, 
\[
S_X(Y)=\frac{\frac{7}{4}H^2+5H\cdot(-K_S)}{H^2+3H\cdot(-K_S)}> \frac{5}{3}>1=A_X(Y),
\]
which implies $\delta(X)<\frac{3}{5}$.

\end{expl}



\begin{bibdiv}
\begin{biblist}[\resetbiblist{BCH\etalchar{+}10+++}\normalsize]

\bib{BCHM}{article}{
      author={Birkar, Caucher},
      author={Cascini, Paolo},
      author={Hacon, Christopher D.},
      author={McKernan, James},
       title={Existence of minimal models for varieties of log general type},
        date={2010},
     journal={J.~Amer.\ Math.\ Soc.},
      volume={23},
      number={2},
      pages={405\ndash 468},
      label={BCH\etalchar{+}10},
}

\bib{BJ-delta}{article}{
      author={Blum, Harold},
      author={Jonsson, Mattias},
       title={Thresholds, valuations, and {K}-stability},
        date={2020},
     journal={Adv.\ Math.},
      volume={365},
       pages={107062},
}

\bib{BX-uniqueness}{article}{
      author={Blum, Harold},
      author={Xu, Chenyang},
       title={Uniqueness of {K}-polystable degenerations of {F}ano varieties},
        date={2019},
     journal={Ann.\ of Math.~(2)},
      volume={190},
      number={2},
       pages={609\ndash 656},
}

\bib{DZ-bigness}{arXiv}{
      author={Darvas, Tam\'as},
      author={Zhang, Kewei},
      title={Twisted {K}\"ahler-{E}instein metrics in big classes},
      date={2022},
       eprint={preprint \arXiv{2208.08324}},
}


\bib{DR-bigness}{arXiv}{
      author={Dervan, Ruadha\'{\i}},
      author={Reboulet, R\'emi},
       title={Ding stability and {K}\"ahler-{E}instein metrics on manifolds
  with big anticanonical class},
        date={2022},
     eprint={preprint \arXiv{2209.08952}},
}

\bib{Fujita-valuative}{article}{
      author={Fujita, Kento},
       title={A valuative criterion for uniform {K}-stability of\, {$\mathbb
  Q$}-{F}ano varieties},
        date={2019},
     journal={J.~reine angew.\ Math.},
      volume={751},
       pages={309\ndash 338},
}


\bib{FO-basistype}{article}{
      author={Fujita, Kento},
      author={Odaka, Yuji},
       title={On the {K}-stability of {F}ano varieties and anticanonical
  divisors},
        date={2018},
     journal={Tohoku Math.\ J.~(2)},
      volume={70},
      number={4},
       pages={511\ndash 521},
}

\bib{Gongyo-weakfano}{article}{
      author={Gongyo, Yoshinori},
       title={On weak {F}ano varieties with log canonical singularities},
        date={2012},
     journal={J.~reine angew.\ Math.},
      volume={665},
       pages={237\ndash 252},
}


\bib{Kol13}{book}{
      author={Koll\'{a}r, J\'{a}nos},
       title={Singularities of the minimal model program \rm{(with a collaboration of S.~Kov\'{a}cs)}},
      series={Cambridge Tracts in Math.},
   publisher={Cambridge Univ.\ Press, Cambridge},
        date={2013},
      volume={200},
}

\bib{KM98}{book}{
      author={Koll\'{a}r, J\'{a}nos},
      author={Mori, Shigefumi},
       title={Birational geometry of algebraic varieties \rm{(with the collaboration of C.\,H.~Clemens and A.~Corti; translated from the 1998 Japanese original)}},
      series={Cambridge Tracts in Math.},
   publisher={Cambridge Univ.\ Press, Cambridge},
        date={1998},
      volume={134},
}

\bib{Lazasfeldbook2}{book}{
      author={Lazarsfeld, Robert},
       title={Positivity in algebraic geometry. {II}},
      series={Ergeb.\ Math.\ Grenzbeg.~(3)},
   publisher={Springer-Verlag, Berlin},
        date={2004},
      volume={49},
        ISBN={3-540-22534-X},
}

\bib{Li-valuative}{article}{
      author={Li, Chi},
       title={K-semistability is equivariant volume minimization},
        date={2017},
     journal={Duke Math.~J.},
      volume={166},
      number={16},
       pages={3147\ndash 3218},
}

\bib{LXZ-HRFG}{article}{
      author={Liu, Yuchen},
      author={Xu, Chenyang},
      author={Zhuang, Ziquan},
       title={Finite generation for valuations computing stability thresholds
  and applications to {K}-stability},
        date={2022},
     journal={Ann.\ of Math.~(2)},
      volume={196},
      number={2},
       pages={507\ndash 566},
}

\bib{Xu-survey}{article}{
      author={Xu, Chenyang},
       title={K-stability of {F}ano varieties: an algebro-geometric approach},
        date={2021},
     journal={EMS Surv.\ Math.\ Sci.},
      volume={8},
      number={1-2},
       pages={265\ndash 354},
}

\bib{XZ-localHRFG}{arXiv}{
      author={Xu, Chenyang},
      author={Zhuang, Ziquan},
       title={Stable degenerations of singularities},
        date={2022},
  eprint={preprint \arXiv{2205.10915}},
}

\end{biblist}
\end{bibdiv}


\end{document}